\documentclass[review,10pt]{elsarticle}
\usepackage{srcltx}
\usepackage{eurosym}
\usepackage{multirow}
\usepackage{mathtools}
\usepackage{amsmath,mathtools}
\usepackage{amsfonts}
\usepackage{amssymb}
\usepackage{amsthm}
\usepackage{graphicx}
\usepackage{mathrsfs}
\usepackage{xcolor}
\usepackage{exscale}
\usepackage{latexsym}
\usepackage{multicol}

\usepackage[backref=page,colorlinks,plainpages=true,pdfpagelabels,
hypertexnames=true,colorlinks=true,pdfstartview=FitV,linkcolor=blue,
citecolor=red,urlcolor=black]{hyperref}
\PassOptionsToPackage{unicode}{hyperref}
\PassOptionsToPackage{naturalnames}{hyperref}
\usepackage{enumerate}
\usepackage[shortlabels]{enumitem}
\usepackage{bookmark}
\usepackage{wasysym}
\usepackage{esint}
\usepackage[ddmmyyyy]{datetime}
\usepackage[margin=2cm]{geometry}
\usepackage{bbm}
\makeatletter
\g@addto@macro\normalsize{%
	\setlength\abovedisplayskip{2pt}
	\setlength\belowdisplayskip{2pt}
	\setlength\abovedisplayshortskip{4pt}
	\setlength\belowdisplayshortskip{4pt}
}
\numberwithin{equation}{section}
\everymath{\displaystyle}

\usepackage[capitalize,nameinlink]{cleveref}
\crefname{section}{Section}{Sections}
\crefname{subsection}{Subsection}{Subsections}
\crefname{subsection}{subsection}{subsections}
\crefname{condition}{Condition}{Conditions}
\crefname{hypothesis}{Hypothesis}{Conditions}
\crefname{assumption}{Assumption}{Assumptions}
\crefname{lemma}{Lemma}{Lemmas}
\crefname{claim}{Claim}{Claims}
\crefname{observation}{Observation}{Observations}
\crefname{example}{Example}{Examples}

\crefformat{equation}{\textup{#2(#1)#3}}
\crefrangeformat{equation}{\textup{#3(#1)#4--#5(#2)#6}}
\crefmultiformat{equation}{\textup{#2(#1)#3}}{ and \textup{#2(#1)#3}}
{, \textup{#2(#1)#3}}{, and \textup{#2(#1)#3}}
\crefrangemultiformat{equation}{\textup{#3(#1)#4--#5(#2)#6}}%
{ and \textup{#3(#1)#4--#5(#2)#6}}{, \textup{#3(#1)#4--#5(#2)#6}}%
{, and \textup{#3(#1)#4--#5(#2)#6}}

\Crefformat{equation}{#2Equation~\textup{(#1)}#3}
\Crefrangeformat{equation}{Equations~\textup{#3(#1)#4--#5(#2)#6}}
\Crefmultiformat{equation}{Equations~\textup{#2(#1)#3}}{ and \textup{#2(#1)#3}}
{, \textup{#2(#1)#3}}{, and \textup{#2(#1)#3}}
\Crefrangemultiformat{equation}{Equations~\textup{#3(#1)#4--#5(#2)#6}}%
{ and \textup{#3(#1)#4--#5(#2)#6}}{, \textup{#3(#1)#4--#5(#2)#6}}%
{, and \textup{#3(#1)#4--#5(#2)#6}}

\crefdefaultlabelformat{#2\textup{#1}#3}
\newtheorem{theorem}{Theorem}[section]
\newtheorem{lemma}[theorem]{Lemma}

\newtheorem{proposition}[theorem]{Proposition}

\newtheorem{definition}[theorem]{Definition}
\newtheorem{remark}[theorem]{Remark}        

\numberwithin{equation}{section}

\makeatletter
\newcommand{\vo}{\vec{o}\@ifnextchar{^}{\,}{}}
\makeatother

\def\YYint#1#2#3{{\setbox0=\hbox{$#1{#2#3}{\iint}$}
    \vcenter{\hbox{$#2#3$}}\kern-.50\wd0}}

\def\XXint#1#2#3{{\setbox0=\hbox{$#1{#2#3}{\int}$}
    \vcenter{\hbox{$#2#3$}}\kern-.50\wd0}}

\makeatletter
\def\namedlabel#1#2{\begingroup
   \def\@currentlabel{#2}%
   \label{#1}\endgroup
}
\makeatother
\makeatletter
\newcommand{\rmh}[1]{\mathpalette{\raisem@th{#1}}}
\newcommand{\raisem@th}[3]{\hspace*{-1pt}\raisebox{#1}{$#2#3$}}
\makeatother

\newcommand{\descref}[2]{\hyperref[#1]{\textnormal{\textcolor{black}{(}\textcolor{blue}{\bf #2}\textcolor{black}{)}}}}

\newcommand{\dref}[2]{\hyperref[#1]{\textcolor{black}{(}\textcolor{blue}{\bf #2}\textcolor{black}{)}}}

\newcommand\RR{\mathbb{R}}

\newcommand{\al}{\alpha}

\newcommand{\ve}{\varepsilon}

\newcommand{\tht}{\theta}

\newcommand{\Om}{\Omega}

\DeclareMathOperator{\diam}{diam}
\DeclareMathOperator{\dist}{dist}

\DeclareMathOperator{\loc}{loc}

\newcommand{\iprod}[2]{\langle #1 \ ,  #2\rangle}

\newcommand{\lbr}[1][(]{\left#1}
\newcommand{\rbr}[1][)]{\right#1}

\newcommand{\txt}[1]{\qquad \text{#1} \qquad}

\newcounter{whitney}
\refstepcounter{whitney}

\newcounter{ineqcounter}
\refstepcounter{ineqcounter}

\makeatletter
\def\ps@pprintTitle{%
\let\@oddhead\@empty
\let\@evenhead\@empty
\def\@oddfoot{}%
\let\@evenfoot\@oddfoot}
\makeatother
\usepackage{setspace}
\usepackage[titletoc,toc,page]{appendix}

\allowdisplaybreaks
\usepackage{pgf,tikz}
\usetikzlibrary{arrows,patterns}
\usetikzlibrary{decorations.pathreplacing}
\nonstopmode
\begin{document}

\begin{frontmatter}

\title{Borderline Lipschitz regularity for bounded minimizers of functionals with $(p,q)$-growth}

\author{Karthik Adimurthi\tnoteref{thanksfirstauthor}}
\ead{karthikaditi@gmail.com and kadimurthi@tifrbng.res.in}

\author{Vivek Tewary\tnoteref{thankssecondauthor}}
\ead{vivek2020@tifrbng.res.in and vivektewary@protonmail.com}

\tnotetext[thanksfirstauthor]{Supported by the Department of Atomic Energy,  Government of India, under
	project no.  12-R\&D-TFR-5.01-0520 and SERB grant SRG/2020/000081}
\tnotetext[thankssecondauthor]{Supported by the Department of Atomic Energy,  Government of India, under
	project no.  12-R\&D-TFR-5.01-0520}
\address{Tata Institute of Fundamental Research, Centre for Applicable Mathematics, Bangalore, Karnataka, 560065, India}

\begin{abstract}
We prove local Lipschitz regularity for bounded minimizers of functionals with nonstandard $p,q$-growth with the source term in the Lorentz space $L(N,1)$ under the restriction $q<p+1+p\,\min\left\{\frac 1N,\frac{2(p-1)}{Np-2p+2}\right\}$. This extends the recent work by Beck-Mingione to bounded minimizers under weaker hypothesis and is sharp for some special ranges of $p$, $q$ and $N$.
\end{abstract}


\end{frontmatter}
\begin{singlespace}
\tableofcontents
\end{singlespace}
\section{Introduction}\label{section0}

Let $\Om\subset\RR^N$ be an open set and consider the following local minimization problem: Find $U\in W^{1,p}_{\loc}(\Om)$ such that for every $\Om'\Subset\Om$, the following problem admits a minimizer:
\begin{align}\label{mainprob}
	\mathfrak{F}(U,\Om'):=\min_{v-U\in W^{1,p}_{0}(\Om')}\mathfrak{F}(v):=\min_{v-U\in W^{1,p}_{0}(\Om')}\int_{\Om'} F(\nabla v)-fv\,dx,
\end{align} where $f\in L(N,1)(\Om)$ and $F\in C^2(\mathbb{R}^N)$ satisfies the following growth and ellipticity conditions: there exists $m,M\in(0,\infty)$ such that for $1< p\leq q<\infty$, $z\in\RR^N$ and $\xi\in\RR^N$, the following is satisfied:
\begin{equation*}\begin{array}{rcl}
	m|z|^p\leq F(z)&\leq& M|z|^p+M|z|^q\\
	|DF(z)|&\leq& M|z|^{p-1}+M|z|^{q-1}\\
	m|z|^{p-2}|\xi|^2\leq \langle D^2F(z)\xi,\xi\rangle&\leq& M|z|^{p-2}|\xi|^2+M|z|^{q-2}|\xi|^2,
	\end{array}
\end{equation*}

\begin{definition}A function $f \in L(N,1)(\Om)$ if the following holds:
	$$f\in L(N,1)(\Om)\coloneqq \left\{ g\in L^n(\Om) : \int_{0}^\infty |\{x\in\Om : |g(x)|>\lambda\, \}|^{1/N}d\lambda<\infty \right\}.$$
\end{definition}

The main theorem we prove in this paper is given next.
\begin{theorem}\label{maintheorem} Let $1< p\leq q<\infty$ with $N\geq 3$ and further assume that the following restrictions are in force:
\begin{align}\label{restriction}
	q<p+1+p\,\min\left\{\frac 1N, \frac{2(p-1)}{Np+2-2p}\right\}\txt{ and } f\in L(N,1)(\Om).
\end{align} 
	 Let $U\in W^{1,p}_{\loc}(\Om)\cap L^\infty_{\loc}(\Om)$ be a bounded local minimizer of $\mathfrak{F}$ as in~\cref{mainprob}, then $\nabla U\in L^\infty_{\loc}(\Om)$ and for all $\alpha>0$ satisfying 
	 \begin{align*}
		\frac{2(q-1)}{p}\leq \alpha <2+\min\left\{\frac{2}{N},\frac{4(p-1)}{Np+2-2p}\right\},
	\end{align*} we have the following estimate
	\begin{equation*}
		||\nabla U||^p_{L^\infty(B_{R/2})}\apprle  \left\{\frac{|B|}{|B_R|}\left(\fint_{B} F(\nabla U) \,dx+||f||^{\frac{p}{p-1}}_{L^N(B)}\right)\right\}^{\frac{2}{2+2N-\alpha N}}
			+1+ ||f||_{L(N,1)(B_R)}^{\frac{2p\alpha}{p(2\alpha-\alpha N+2N)-2\alpha}},
		\end{equation*} for some ball $B$ such that $B_R\Subset B\Subset \Omega$ and for some constant $C$ depending on $N$, $p$, $||U||_{L^\infty(2B)},||f||_{L^N(B)}$, $M$ and $m$
\end{theorem}

\subsection{Comparision to Previous Results}
 This proof uses the techniques in Beck-Mingione~\cite{beckLipschitzBoundsNonuniform2020}, where the same theorem was proved for unbounded minimizers under the restriction
 \begin{align}\label{prevrestriction}
	\frac{q}{p}<1+\min\left\{\frac{2}{N},\frac{4(p-1)}{p(N-2)}\right\}\txt{ and } f\in L(N,1)(\Om),\,N\geq 3.
\end{align}
The bound~\cref{prevrestriction} is an extension of the classical results of Marcellini~\cite{marcelliniRegularityMinimizersIntegrals1989,marcelliniRegularityExistenceSolutions1991}, who has made some of the first contributions in the regularity theory of problems of nonstandard growth in the Western world. There have been parallel contributions from the Soviet school~\cite{kolodiuiBoundednessGeneralizedSolutions1971,kolodiuiEstimateMaximumModulus1974,uraltsevaBoundednessGradientsGeneralized1983}. In a very nice recent paper, P.Bella and M.Sch\"affner \cite{bellaLipschitzBounds2022} improved the restriction to
\begin{align}\label{bellaschaffner}
	\frac{q}{p}<1+\min\left\{\frac{2}{N-1},\frac{4(p-1)}{p(N-3)}\right\}, \quad \mbox{ for } N\geq 3 \quad \text{and} \quad p >1,
\end{align} by employing a specialized test function that enables them to use Sobolev embedding on the sphere. 
There is a large body of work dealing with problems of $(p,q)$-growth as well as other nonstandard growth problems, for which we refer to the detailed surveys in \cite{marcelliniRegularityGeneralGrowth2020,mingioneRecentDevelopmentsProblems2021}.

It is well known that Lipschitz continuity and even boundedness for \cref{mainprob} fail when $p$ and $q$ are far apart as evidenced by the following example of Hong \cite{hongRemarksMinimizersVariational1992}, which is a variation on the famous counterexample of Giaquinta \cite{giaquintaGrowthConditionsRegularity1987}:
\begin{align*}
	\int_{\Om} |\nabla u|^2+|u_{x_n}|^4\,dx,
\end{align*} which satisfies the hypothesis for $p=2$, $q=4$ and $f\equiv 0$ and admits an unbounded minimizer if $N\geq 6$ (more examples of unbounded minimizers of \cref{mainprob} may be found in \cite{marcelliniExempleSolutionDiscontinue1987}). It was shown in \cite[Section~6]{marcelliniRegularityExistenceSolutions1991} that if $q>\frac{(N-1)p}{N-1-p}$, then one cannot expect boundedness and only recently, this restriction was found to be sharp in \cite{hirschGrowthConditionsRegularity2020}, where it is proved that the minimizer is bounded provided
\begin{align}\label{hirschschaffner}
	\frac{1}{p}-\frac{1}{q}\leq \frac{1}{N-1}.
\end{align}

The aim of this article is to employ the technique of integration by parts which is already used by DiBenedetto~\cite{dibenedettoRegularitySolutionsNonlinear1984} and has been used to great effect in the works of Bildhauer and Fuchs~\cite{bildhauerInteriorRegularityFree2004}. Recently, the same technique was used very effectively by Bousquet and Brasco~\cite{bousquetLipschitzRegularityOrthotropic2020} to prove Lipschitz regularity for bounded minimizers of the anisotropic functional with nonstandard growth conditions without imposing any restrictions on the exponents.

It is easy to see that there is a  gap between the restrictions in \cref{hirschschaffner} and \cref{bellaschaffner} and in this context, the authors in \cite{bellaRegularityMinimizersScalar2020,mingioneRecentDevelopmentsProblems2021} asked if one could obtain a Sobolev-type restriction (as in \cref{hirschschaffner}) in order for the minimizer to be  Lipschitz regular. \emph{In this regard, we improve the restriction in \cref{bellaschaffner} in some special ranges of $p,q$ and $N$ and also partially provide an answer to the question from \cite{bellaRegularityMinimizersScalar2020,mingioneRecentDevelopmentsProblems2021} by obtaining a Sobolev type restriction when $N>\frac 12\left(p^2 + \sqrt{p^4 + 6 p^2 + 4 p + 1}+ 1\right)$ and $p\geq 2$.} 

We make the following observations regarding the sharpness of the results. Since the thresholds of sharpness are sometimes cumbersome to state, we only point them out in special ranges:

\begin{enumerate}[(i)]
    \item For bounded minimizers, we require $q < p+1+p\,\min\left\{\frac 1N, \frac{2(p-1)}{Np+2-2p}\right\}$  and $f\in L(N,1)(\Om)$, see \cref{maintheorem}. For $p > \frac{2(N+1)}{(N+2)}$, and therefore, in particular, for $p\geq 2$, the restriction for bounded minimizers simplifies to \[q < p+1+\frac pN.\] On the other hand, for $p<\frac{2(N+1)}{(N+2)}$, the restriction for bounded minimizers simplifies to \[q < p+1+\frac{2p(p-1)}{Np+2-2p}.\]
    \item Combining the restriction $q < p+1+p\,\min\left\{\frac 1N, \frac{2(p-1)}{Np+2-2p}\right\}$ with the optimal restriction for boundedness from \cref{hirschschaffner}, we see that Lipschitz regularity for minimizers holds provided  $$\frac{q}{p} < 1 + \min\left\{\frac{p}{N-1-p},\frac 1N+\frac{1}{p},\frac{2(p-1)}{Np+2-2p} \right\},$$ and $f\in L(N,1)(\Om)$. 
    \item In the case $2\leq p\leq \frac{N(N-1)}{N+1}$, we see that $\frac{1}{p}+\frac{1}{N} \geq \frac{2}{N-1}$, which suggests that \cref{maintheorem} improves the restriction given in \cref{bellaschaffner} in this range. But it must be noted that our result additionally requires that the solutions are bounded which also requires the restriction \cref{hirschschaffner} to be satisfied.
    \item {\underline{\bfseries Case: $p\geq \frac{2(N+1)}{(N+2)}$:}} Since we require bounded solutions, we see that for minimizers, Lipschitz regularity would then require $\frac{q}{p} < 1+\min\left\{\frac{p}{N-1-p},\frac{1}{p}+\frac 1N \right\}$. In particular, if $N>\frac 12\left(p^2 + \sqrt{p^4 + 6 p^2 + 4 p + 1}+ 1\right)$ and $p\geq \frac{2(N+1)}{(N+2)}$, then $\frac{1}{p} + \frac 1N> \frac{p}{N-1-p}$  and thus Lipschitz regularity holds for any minimizer as they are automatically bounded. In particular, due to the sharpness of the condition \cref{hirschschaffner}, we automatically obtain sharpness of the Lipschitz regularity in this range.
    \item {\underline{\bfseries Case: $p < \frac{2(N+1)}{(N+2)}$:}} Once again, since we require bounded solutions, we see that for minimizers, Lipschitz regularity would then require $\frac{q}{p} < 1+\min\left\{\frac{p}{N-1-p}, \frac 1p + \frac{2(p-1)}{Np+2-2p} \right\}$. In particular, if $\frac{1}{p} + \frac{2(p-1)}{Np+2-2p} > \frac{p}{N-1-p}$  then Lipschitz regularity holds for any minimizer as they are automatically bounded. In particular, due to the sharpness of the condition \cref{hirschschaffner}, we automatically obtain sharpness of the Lipschitz regularity in this range. However, in this case, an explicit condition on $N$ looks unwieldy and is omitted.
    \item Our theorem improves the previous restriction for bounded minimizers of \cref{mainprob} which was found to be $q<p+1$ for $f\equiv 0$ in \cite{choeInteriorBehaviourMinimizers1992,cupiniExistenceRegularityElliptic2014}. 
\end{enumerate}

Let us now briefly describe the method of proof: In \cref{section1}, we start with the notations and auxiliary results that will be used in the course of the paper. In \cref{section2}, we describe a regularization procedure noting that the regularized solution belongs to $W^{1,\infty}\cap W^{2,2}$. In \cref{section3}, we obtain an energy estimate to be used for a De Giorgi-type iteration. In \cref{section4}, we apply the integration by parts technique and make use of restrictions on $q$ to get an improved Caccioppoli inequality. At this point, in \cref{section5}, we apply the De Giorgi iteration to obtain a boundedness estimate for $\nabla U$. This is followed followed by an interpolation estimate which results in the quoted restriction~\cref{restriction} (see \cref{section6}). Finally, we pass to the limit in the regularization parameter in \cref{section7}.

\section{Notations and Preliminaries}\label{section1}
\subsection{Notations}
We begin by collecting the standard notation that will be used throughout the paper.
\begin{itemize}
	\item We shall denote $N$ to be the space dimension. A point in $\mathbb{R}^{N}$ will be denoted by $x$. 
	\item Let $\Omega$ be a domain in $\mathbb{R}^N$ of boundary $\partial \Omega$.
	\item The notation $a \lesssim b$ is shorthand for $a\leq C b$ where $C$ is a constant independent of the regularization parameters $\sigma$ and $\ve$ and depends only on the data. 
\end{itemize}
\subsection{Preliminaries for Regularization}
We list some of the preliminaries that are required in the subsequent sections. Bousquet and Brasco~\cite{bousquetGlobalLipschitzContinuity2016} have proved a general theorem on the local Lipschitz regularity of minimizers to convex minimization problems posed on convex domains with boundary values satisfying the {\itshape{bounded slope condition}}. It is important to note that such a theorem does not require any growth condition on the functional. Let us recall the following definitions.

\begin{definition}
	A bounded, open set $B\subset\mathbb{R}^N$ is said to be uniformly convex if there exists $\nu>0$ such that for every boundary point $x_0\in\partial B$ there exists a hyperplane $H_{x_0}$ passing through that point satisfying
	\begin{align*}
		\dist(y,H_{x_0})\geq \nu |y-x_0|^2,\mbox{ for any $y\in\partial B$.}
	\end{align*}
\end{definition}

\begin{definition}[Bounded Slope Condition]
	Let $K$ be a positive real number and $B$ an open bounded convex subset of $\mathbb{R}^N$. We say that a function $\phi:\partial B\to\mathbb{R}$ satisfies the bounded slope condition of rank $K$ if for any $x_0\in\partial B$ there exists vectors $l_{x_0}^-$ and $l_{x_0}^+$ satisfying $||l_{x_0}^-||\leq K, ||l_{x_0}^+||\leq K$ such that
	\begin{align*}
		l_{x_0}^-\cdot(x-x_0)\leq \phi(x)-\phi(x_0)\leq l_{x_0}^+\cdot(x-x_0),\mbox{ for any $x\in\partial B$.}
	\end{align*}
\end{definition}

The following proposition gives a sufficient condition for a function to satisfy the bounded slope condition, the proof of which can be found in \cite[p. 234, Theorem 1.2]{mirandaTeoremaDiEsistenza1965}.

\begin{proposition}\label{characterbsc}
	Let $B\subset\mathbb{R}^N$ be a uniformly convex domain. Then, any function $\phi\in C^2(\mathbb{R}^N)$ satisfies the bounded slope condition on $\partial B$.
\end{proposition}

We finally state the theorem of Bousquet and Brasco on Lipschitz regularity of convex minimization problems, see \cite[p. 1406, Main Theorem]{bousquetGlobalLipschitzContinuity2016} for the details.
\begin{theorem}\label{lipboubra}
	Let $G\subset \mathbb{R}^N$ be a bounded open convex set, $\phi:\mathbb{R}^N\to\mathbb{R}$ a Lipschitz continuous function, $H:\mathbb{R}^N\to\mathbb{R}$ a convex function and $g\in L^\infty(G)$. Consider the following problem
	\begin{equation*}
		\inf\left\{ \mathcal{H}(u)\coloneqq \int_{G} H(\nabla u)-gu\,dx: u-\phi\in W_0^{1,1}(G) \right\}.
	\end{equation*} 
	Assume that $\phi|_{G}$ satisfies the bounded slope condition of rank $K>0$ and that $H$ satisfies
	\begin{align}\label{stricon}
		\theta H(z) + (1-\theta)H(z')-H(\theta z + (1-\theta)z')\geq c\theta(1-\theta)(|z|+|z'|)^{p-2}|z-z'|^2,
	\end{align} for all $z,z'\in\mathbb{R}^N$, for some $c>0$ and for all $\theta\in [0,1]$. Then the minimization problem admits at least one solution and every solution is Lipschitz continuous.
\end{theorem}

\subsection{Preliminaries on Lorentz spaces and Nonlinear Potentials}
We follow the notation of~\cite{beckLipschitzBoundsNonuniform2020} for the potential theoretic quantities of interest. Let us define the modified nonlinear potential of $g\in L^2_{\loc}(\mathbb{R}^N)$ as
\begin{equation*}
\mathbf{P}_1^g(x_0,R)\coloneqq \int_0^R\left(\rho^2\fint_{B_{\rho}(x_0)}|g(x)|^2\,dx\right)^{1/2}\frac{d\rho}{\rho},
\end{equation*}
for some $x_0\in\mathbb{R}^N$ and $R>0$. In subsequent sections, we need to know the action of the potential on Lorentz space $L(N,1)$, more specifically, we require the following lemma, whose proof may be found in~\cite[Lemmas~2.3-2.4]{kuusiPotentialEstimatesGradient2012}.
\begin{lemma}
	Let $g\in L^2_{\loc}(\mathbb{R}^N)$, then on any $B_R(x_0)\in\mathbb{R}^N$, we have the following bound:
	\begin{equation*}
		||\mathbf{P}_1^g(\cdot,R)||_{L^\infty(B_R)}\leq c ||g||_{L(N,1)(B_{2R})}.
	\end{equation*}
\end{lemma}

\subsection{Some well-known results}

We shall make use of the following well-known iteration lemma whose proof may be found in~\cite[Lemma 6.1]{giustiDirectMethodsCalculus2003}.

\begin{lemma}\label{iterlemma}
	Let $Z(t)$ be a bounded non-negative function in the interval $\rho,R$. Assume that for $\rho\leq t<s\leq R$ we have
	\begin{equation*}
		Z(t)\leq [A(s-t)^{-\alpha}+B(s-t)^{-\beta}+C]+\vartheta Z(s),
	\end{equation*} with $A,B,C\geq 0$, $\alpha,\beta>0$ and $0\leq \vartheta<1$. Then,
\begin{align*}
	Z(\rho)\leq c(\alpha,\vartheta)[A(R-\rho)^{-\alpha}+B(R-\rho)^{-\beta}+C].
\end{align*}
\end{lemma}

We shall also need the following standard uniqueness result which can be found in \cite[Proposition 2.10]{rindlerCalculusVariations2018}.
	\begin{theorem} \label{uniqueness}
		Let $\mathcal{H}:W^{1,r}(\Omega)\to\mathbb{R},\, r\in [1,\infty)$, be an integral functional with a $C^2$ integrand $H:\mathbb{R}^N\to\mathbb{R}$. If $H$ is strictly convex, that is \[H(\theta z_1+(1-\theta)z_2)<\theta H(z_1)+(1-\theta)H(z_2),\] for all $z_1,z_2\in\mathbb{R}^N$ with $z_1\neq z_2$, $\theta\in (0,1)$, then the minimizer $u_*\in W^{1,r}_g(\Omega)=\{u\in W^{1,r}(\Omega):u|_{\partial\Omega}=g\}$, where $g\in W^{1-1/r,r}(\partial \Omega)$, of $\mathcal{H}$, if it exists, is unique.
	\end{theorem}

We end this subsection by recalling a maximum principle, whose proof may be adapted from~\cite[Chapter 4, Theorem 4.1.2]{wuEllipticParabolicEquations2006}.

\begin{theorem}[Maximum Principle]\label{maxprince}
	Let $p\in (1,\infty)$. Let $a_{ij}(x,z), i,j=1,2,\ldots,N$ be measurable and bounded functions for $x\in B$ and $C^2$ for $z\in\RR^N$ such that 
	\begin{align*}
		\mu_1|z|^{p-2}|\xi|^2\leq a_{ij}(x,z)\xi_i\xi_j, \mbox{ a.e. } x\in B, \mbox{ for all }\xi,z\in\mathbb{R}^N,
	\end{align*} and $f\in L^s(B)$ for $s>N/p$. If $u\in W^{1,p}(B)$ satisfies
	\begin{align*}
		\int_B a_{ij}(x,z)u_{x_i}(x)v_{x_j}(x)\,dx=\int_B\,f\,v\,dx,\mbox{ for all }v\in H^1_0(B),
	\end{align*} then we have
	\begin{align*}
		||u||_{L^\infty(B)}\leq \max_{x\in\partial B} u(x)+C||f||_{L^s(B)}.
	\end{align*}
\end{theorem}



\section{Regularization}\label{section2}

\subsection{Approximation Scheme}
\label{approxscheme}

Let us fix a ball $4B\Subset\Om$ and take $\ve_0=\min\left\{\frac{1}{4},\frac{\diam(B)}{2}\right\}>0$. Now, we take a sequence $\{\ve_n\}$ of positive numbers such that $\ve_n\to 0$ as $n\to\infty$ and $\ve_n\in(0,\ve_0)$ for all $n\in\mathbb{N}$. Using a standard mollifier $\rho_{\ve_n}(z)\coloneqq \ve^{-N}\rho(z/\ve_n)$, where $\rho\in C^{\infty}(\mathbb{R}^N)$ is such that
\begin{align*}
	\text{supp}\,\rho=\overline{B_1(0)}\mbox{ and }\int_{\mathbb{R}^N}\rho(z)\,dz=1,
\end{align*}
we define $U_{n}\coloneqq U\ast\rho_{\ve_n}$.

Furthermore, we  define the regularized functional
\begin{align}\label{mainprob2}
	\mathfrak{F}_{n}(w):=\int_{\Om} F_{n}(\nabla w)-f_{n}w\,dx,
\end{align} 
where $F_{n}(z)=F\ast\rho_{\ve_n}(z)\in C^\infty(\mathbb{R}^N)$ satisfies the following growth and ellipticity conditions with some new constants $m_0=m_0(m,M,p,q,N)$ and $M_0=M_0(m,M,p,q,N)$ (see \cite[Lemma 3.1]{espositoRegularityResultsMinimizers2002} for more details): for $1 < p\leq q<\infty$, $z\in\RR^N$ and $\xi\in\RR^N$, we assume the following holds:
\begin{subequations}
\begin{equation}\label{rhyp3}
	\begin{array}{rcl}
	m_0({\ve_n}^2+|z|^2)^\frac{p}{2}\leq F_{n}(z)&\leq& M_0({\ve_n}^2+|z|^2)^{\frac{p}{2}}+M_0({\ve_n}^2+|z|^2)^{\frac{q}{2}},\\
	|DF_{n}(z)|&\leq& M_0({\ve_n}^2+|z|^2)^{\frac{p-1}{2}}+ M_0({\ve_n}^2+|z|^2)^{\frac{q-1}{2}},\\
	m_0({\ve_n}^2+|z|^2)^{\frac{p-2}{2}}|\xi|^2\leq \langle D^2F_{n}(z)\xi,\xi\rangle&\leq& M_0({\ve_n}^2+|z|^2)^{\frac{p-2}{2}}|\xi|^2+M_0({\ve_n}^2+|z|^2)^{\frac{q-2}{2}}|\xi|^2,
	\end{array}
\end{equation}
and $f_{n}(x)\coloneqq \min\{\max\{f(x),-n\},n\}$. It is easy to see that  $f_n\in L^\infty(B)$ and $|f_n|\leq \min\{|f|,n\}$.

\end{subequations} We denote $u_{n}$ to be the minimizer of $\mathfrak{F}_{n}$ in $B$, i.e.,
\begin{align}\label{regminim}
	\mathfrak{F}_{n}(u_{n})=\min_{v\in U_{n}+W^{1,p}_0(B)}\int_{\Om} F_{n}(\nabla v)-f_{n}v\,dx.
\end{align}

\subsection{Existence and Regularity of minimizers}

\begin{lemma}\label{reglemma}
	There exists a unique $u_{n}\in U_{n}+W^{1,1}_0(B)$ satisfying~\cref{regminim}. In particular, we have
\begin{align*}
	\mathfrak{F}_{n}(u_{n})\leq \mathfrak{F}_{n}(v),
\end{align*} for all $v\in U_{n}+W^{1,p}_0(B)$. 
	Moreover, $u_{n}\in W^{1,\infty}(B)\cap W^{2,2}(B)$ and for all $0<\ve_n<\ve_0$, the following holds:
	$$||u_n||_{L^\infty(B)}\leq||U||_{L^\infty(2B)}+C||f||_{L^N(B)}.$$
\end{lemma}

\begin{proof}
	Since $U_n\in C^2(\overline{B})$, we see that $U_n$ satisfies the bounded slope condition. Moreover, by~\cref{rhyp3}, $F_n$ satisfies the convexity condition~\cref{stricon}. Therefore, by~\cref{lipboubra}, we have the existence of a minimizer $u_{n}$ such that $u_{n}\in W^{1,\infty}(B)$. The condition~\cref{stricon} implies strict convexity of $F_n$ and hence, $u_{n}$ is a unique minimizer of $\mathcal{F}_{n}$ by~\cref{uniqueness}. 
	
	Let us denote the number $l_n:=||\nabla u_n||_{L^\infty(B)}<\infty$. Clearly the function $u_{n}$ satisfies the following Euler-Lagrange equation
	\begin{equation*}
		\nabla\cdot(\nabla F_{n}(\nabla u_{n}))=f_n	\txt{in} B.\end{equation*}
 By~\cref{rhyp3}, we have
\begin{align*}
	m_0\ve_n^{p-2}|\xi|^2\leq \langle D^2F_n(\nabla u_{n})\xi,\xi\rangle\leq M_0\left[(\ve_n^2+l_n^2)^{\frac{p-2}{2}}+(\ve_n^2+l_n^2)^{\frac{q-2}{2}}\right]|\xi|^2. 
\end{align*}Hence, by a standard argument involving difference quotients, we can prove that $u_{n}\in W^{2,2}(B)$. 

To prove the required estimate, we invoke the Maximum principle from \cref{maxprince} along with the lower bound from \cref{rhyp3} to conclude that
\begin{align*}
	\max_{B}|u_n|\leq \max_{\partial B}|u_n|+C||f||_{L^N(B)}= \max_{\partial B}|U_{n}|\leq \max_{2B}|U|+C||f||_{L^N(B)}.
\end{align*}
Note that the $L^{\infty}$ estimate is uniform and independent of $n$ and  $\ve_n$.
\end{proof}

\begin{remark}
	Indeed, the $W^{1,\infty}\cap W^{2,2}$ estimates in~\cref{reglemma} depend on $n$ and $\ve_n$. However, in subsequent sections,  we only require the qualitative fact that $u_n\in W^{1,\infty}(B)\cap W^{2,2}(B)$ in order to rigorously justify all the calculations.
\end{remark}

\begin{remark}In subsequent sections excepting the final two sections, we shall suppress writing the subscript of $u_{n}, F_n, f_n$ for ease of notation. Moreover, in place of $\ve_n$, we will write $\ve$.\end{remark}

\section{Caccioppoli Inequality}\label{section3}

In this section, we derive a Caccioppoli inequality needed to apply De Giorgi iteration. Let us first define the following the notation:
\begin{equation}
	\label{eq4.1}
	v  := (\ve^2+|\nabla u|^2)^{p/2},\quad 
	g_1(t)  :=  (\ve^2+t^2)^{{(p-2)}/2} \txt{and}
	g_2(t)  :=  (\ve^2+t^2)^{{(p-2)}/2} + (\ve^2+t^2)^{{(q-2)}/2}.
\end{equation}

\begin{proposition}\label{caccioppoli2}
	Let $u$ be the solution to~\cref{regminim} and let us take $0<\rho<r<\infty$ such that the ball $B_r(x_0)\Subset B$, for some $x_0\in B$ (recall that $B$ is a fixed ball as defined in \cref{approxscheme}). Let us denote the Lipschitz bound by $L$, i.e., the following bound holds $||Du||_{L^\infty(B_r(x_0))}\leq L$. Then, for each $k\geq 0$, the following estimate holds:
	\begin{equation}\label{cacc2}
		\begin{array}{rcl}
		\int_B g_1(|\nabla u|)|\nabla^2 u|^2 (v-k)_{+}\eta^4\,dx+\int_B|\nabla (v-k)_{+}|^2\eta^4\,dx &\leq &  \frac{C}{(r-\rho)^2}\int_B \frac{g_2(|\nabla u|)}{g_1(|\nabla u|)}  (v-k)^2_{+} \eta^2\,dx \\
		&& +C L^2\int_B |f|^2 |\eta|^4\,dx,
		\end{array}
			\end{equation} for a constant $C$ independent of $L,k,\ve$ where $\ve$ is from \cref{eq4.1}. Without loss of generality, we have assumed that $\ve \leq L$.
\end{proposition}

\begin{proof}
	The minimizer $u$ of~\cref{regminim} satisfies the following Euler-Lagrange equation:
	\begin{align*}
		\int_B\langle DF(\nabla u),\nabla \phi\rangle\,dx=\int_B f\phi\,dx.
	\end{align*} By choosing $\phi=\psi_{x_j}\in H^1_0(B)$ and integrating by parts, we get
	\begin{align*}
	\int_B\langle D^2f_{\sigma}(\nabla u)\nabla u_{x_j},\nabla \psi\rangle\,dx=-\int_B f\psi_{x_j}\,dx.
	\end{align*} Now, we choose $\psi=u_{x_j}(v-k)_{+}\eta^4$, where $\eta\in C_0^\infty(B)$ such that $0\leq \eta\leq 1$ in $B_r(x_0)$,  $\eta\equiv1$ in $B_{\rho}(x_0)$ and $|\nabla \eta|\leq C/(r-\rho)$ and after summing over $j\in\{1,2,\ldots,N\}$, we get
	\begin{align*}
		\underbrace{\sum_{j=1}^N\int_B\left\langle D^2F(\nabla u)\nabla u_{x_j},\nabla u_{x_j}\right\rangle (v-k)_{+} \eta^4\,dx}_{I}&+\underbrace{\sum_{j=1}^N\int_B\left\langle D^2F(\nabla u)\nabla u_{x_j},\nabla (v-k)_{+}\right\rangle u_{x_j}\eta^4\,dx}_{II}\nonumber\\
		&=\underbrace{-4\sum_{j=1}^N\int_B\left\langle D^2F(\nabla u)\nabla u_{x_j},\nabla \eta\right\rangle u_{x_j}\eta^3 (v-k)_{+}\,dx}_{III}-\underbrace{\sum_{j=1}^N\int_B f\psi_{x_j}\,dx}_{IV}.
	\end{align*}
	We substitute $\nabla (v-k)_{+} =\nabla v= p\,g_1(|\nabla u|)\sum_{j=1}^N u_{x_j}\nabla u_{x_j}$, which holds whenever $v>k$, in $II$ and $III$ to obtain:
	\begin{align}\label{step6andhalf}
		\sum_{j=1}^N&\int_B\left\langle D^2F(\nabla u)\nabla u_{x_j},\nabla u_{x_j}\right\rangle (v-k)_{+} \eta^4\,dx+{\int_B\frac{1}{p\, g_1(|\nabla u|)}\left\langle D^2F(\nabla u)\nabla (v-k)_{+}
		,\nabla (v-k)_{+}\right\rangle\eta^4\,dx}\nonumber\\
		&={-4\int_B\frac{1}{p\,g_1(|\nabla u|)}\left\langle D^2F(\nabla u)\nabla (v-k)_{+},\nabla \eta\right\rangle\eta^3 (v-k)_{+}\,dx}-{\sum_{j=1}^N\int_B f\psi_{x_j}\,dx}.
	\end{align}
	Now observe that due to the coercivity of $D^2f_{\sigma}(z)$, we have
	\begin{align*}
		\langle D^2F(\nabla u)\nabla u_{x_j},\nabla \eta\rangle\leq\langle D^2F(\nabla u)\nabla u_{x_j},\nabla u_{x_j}\rangle^{1/2}\langle D^2F(\nabla u)\nabla \eta,\nabla \eta\rangle^{1/2},
	\end{align*} which follows from a Cauchy-Schwartz inequality $\langle Ax,y\rangle\leq \langle Ax,x\rangle^{1/2}\langle Ay,y\rangle^{1/2}$. 
	
	As a consequence, by an application of Young's inequality, for the right hand side (subsequently RHS), we obtain
	\begin{equation}\label{step7}
		III\leq \frac{1}{2}\int_B\frac{1}{p\, g_1(|\nabla u|)}\left\langle D^2F(\nabla u)\nabla (v-k)_{+},\nabla (v-k)_{+}\right\rangle \eta^4\,dx+2\int_B\frac{1}{p\, g_1(|\nabla u|)}\langle D^2F(\nabla u)\nabla \eta,\nabla\eta\rangle (v-k)_{+}^2\eta^2\,dx.
	\end{equation} 
	
	Substituting~\cref{step7} in~\cref{step6andhalf}, we get
	\begin{align*}
		{\sum_{j=1}^N\int_B\left\langle D^2F(\nabla u)\nabla u_{x_j},\nabla u_{x_j}\right\rangle (v-k)_{+} \eta^4\,dx}&+{\int_B\frac{1}{2p\, g_1(|\nabla u|)}\left\langle D^2F(\nabla u)\nabla (v-k)_{+},\nabla (v-k)_{+}\right\rangle\eta^4\,dx}\nonumber\\
		&\leq{2\int_B\frac{1}{p\,g_1(|\nabla u|)}\left\langle D^2F(\nabla u)\nabla \eta,\nabla \eta\right\rangle (v-k)_{+}^{2}\eta^2\,dx}-{\sum_{j=1}^N\int_B f\psi_{x_j}\,dx}.
	\end{align*}
	
	Now, we shall apply~\cref{rhyp3} and estimate from below to get
	\begin{align}
		m_0\int_B g_1(|\nabla u|)|\nabla^2 u|^2 (v-k)_{+}\eta^4\,dx&+\frac{m_0}{2p}\int_B|\nabla (v-k)_{+}|^2 \eta^4\,dx\nonumber\\
		&\leq \frac{2M_0}{p}\int_B \frac{g_2(|\nabla u|)}{g_1(|\nabla u|)} (v-k)_{+}^{2} |\nabla \eta|^2\eta^2\,dx-{\sum_{j=1}^N\int_B f\psi_{x_j}\,dx}\label{step8}.
	\end{align}
	
	For suitably chosen $\sigma_1,\sigma_2\in\RR$, the term involving $f$ is estimates as follows:
	\begin{align}
		{\sum_{j=1}^N\int_B f\psi_{x_j}\,dx}\leq &\sigma_1\int_B g_1(|\nabla u|)|\nabla^2 u|^2\,(v-k)_{+}\eta^4\,dx+\sigma_2\int_B |\nabla (v-k)_{+}|^2\eta^4\,dx\nonumber\\
		&+C\int_B (v-k)_{+}^{2}|\nabla \eta|^2\eta^2\,dx\nonumber\\
		&+C\int_B|f|^2\left\{ [g_1(|\nabla u|)]^{-1} (v-k)_{+}+|\nabla u|^2 \right\}\eta^4\,dx.\label{step9}
	\end{align}
	Combining~\cref{step8} and~\cref{step9}, we obtain
		\begin{align*}
		\frac{m_0}{2}\int_B g_1(|\nabla u|)|\nabla^2 u|^2 (v-k)_{+}\eta^4\,dx&+\frac{m_0}{4p}\int_B|\nabla (v-k)_{+}|^2\eta^4\,dx\nonumber\\
		&\leq C\int_B \frac{g_2(|\nabla u|)}{g_1(|\nabla u|)} (v-k)_{+}^{2} |\nabla \eta|^2\eta^2\,dx+C \int_B |f|^2 (\ve^2 +|\nabla u|^2)\eta^4\,dx,
	\end{align*}
	where we have used the facts that $\dfrac{g_2(|\nabla u|)}{g_1(|\nabla u|)}\geq 1$ and $g_1(|\nabla u|)^{-1}(v-k)_{+}\leq C(\ve^2 + |\nabla u|^2)$.
	This finishes the proof of the proposition.
\end{proof}
\section{Improvement through integration by parts}\label{section4}
In this section, we will apply integration by parts to the RHS of~\cref{caccioppoli2} to obtain an improved Caccioppoli inequality.

\begin{proposition}\label{caccioppoli3}
	Let $u$ be the solution to~\cref{regminim}. Let $0<\rho<r<\infty$ such that the ball $B_r(x_0)\Subset B$, for $x_0\in B$. Let $L$ such that $||Du||_{L^\infty(B_r(x_0))}\leq L$,  then, for each $k\geq 0$, it holds that
	\begin{align}\label{cacc3}
		\int_{B_{\rho}(x_0)}|\nabla (v-k)_{+}|^2 \eta^4\,dx\leq \frac{C}{(r-\rho)^4}\left(\int_{B_r(x_0)} v^{\al} \chi_{\{v>k\}}\,dx+1\right)+C L^2\int_{B_r(x_0)} |f|^2\,dx,
			\end{align} for a constant $C$ depending on $||U||_{L^\infty(2B)},||f||_{L^N(B)}$ but independent of $L,k,\ve$. Here we have taken $\al$ satisfying 
		\begin{equation}\label{eq5.7}
			\al \geq \left\{\frac{q+p-2}{p},\frac{2(q-1)}{p}\right\}.
		\end{equation}
\end{proposition}

\begin{proof}
	From \cref{cacc2}, we have
 \begin{equation}\label{cacc23}
 	\begin{array}{rcl}
 		\int_B g_1(|\nabla u|)|\nabla^2 u|^2 (v-k)_{+}\eta^4\,dx+\int_B|\nabla (v-k)_{+}|^2\eta^4\,dx &\leq &  \frac{C}{(r-\rho)^2}\int_B \frac{g_2(|\nabla u|)}{g_1(|\nabla u|)}  (v-k)^2_{+} \eta^2\,dx \\
 		&& +C L^2\int_B |f|^2 |\eta|^4\,dx,
 	\end{array}
 \end{equation}
We estimate the first term appearing on the right hand side of \cref{cacc23} by 
	\begin{align}\label{est1}
		\frac{1}{(r-\rho)^2}\int_B \frac{g_2(|\nabla u|)}{g_1(|\nabla u|)}  (v-k)^2_{+} \eta^2\,dx&=\frac{1}{(r-\rho)^2}\int_B (1+(\ve^2+|\nabla u|^2)^{\frac{q-p}{2}})  (v-k)^2_{+} \eta^2\,dx\nonumber\\
		&=\frac{1}{(r-\rho)^2}\int_B (v-k)^2_{+} \eta^2\,dx+\frac{1}{(r-\rho)^2}\int_B (\ve^2+|\nabla u|^2)^{\frac{q-p}{2}} (v-k)^2_{+} \eta^2\,dx.
	\end{align}
	The first term on the right hand side of  \cref{est1}  is in the correct form and does not require further analysis. For the second term on the right hand side of \cref{est1}, we expand to get
	\begin{align}\label{est2}
		\frac{1}{(r-\rho)^2}\int_B (\ve^2+|\nabla u|^2)^{\frac{q-p}{2}} (v-k)^2_{+} \eta^2\,dx=&\frac{1}{(r-\rho)^2}\int_B \ve^2(\ve^2+|\nabla u|^2)^{\frac{q-p-2}{2}} (v-k)^2_{+} \eta^2\,dx\nonumber\\
		&+\frac{1}{(r-\rho)^2}\int_B |\nabla u|^2(\ve^2+|\nabla u|^2)^{\frac{q-p-2}{2}} (v-k)^2_{+} \eta^2\,dx
	\end{align}

	The first term in~\cref{est2} is estimated as 
	\begin{equation*}
		\begin{array}{rcl}
		\frac{1}{(r-\rho)^2}\int_B \ve^2(\ve^2+|\nabla u|^2)^{\frac{q-p-2}{2}} (v-k)^2_{+} \eta^2\,dx&\leq & \frac{1}{(r-\rho)^2}\int_B \ve^2(\ve^2+|\nabla u|^2)^{\frac{q-p-2}{2}} v^2\,\chi_{\{v>k\}} \eta^2\,dx\\
		&=&\frac{1}{(r-\rho)^2}\int_B \ve^2(\ve^2+|\nabla u|^2)^{\frac{q+p-2}{2}}\,\chi_{\{v>k\}} \eta^2\,dx\\
		&\leq & \frac{C}{(r-\rho)^2}\left(\int_B v^{\al}\chi_{\{v>k\}} + 1\,dx\right),
\end{array}	\end{equation*} where we have used H\"older inequality with exponent $\frac{p\al}{q+p-2} \geq 1$ using the choice of $\al$ from \cref{eq5.7}.
	
	The second term in~\cref{est2} is estimated by using integration by parts as follows:
	\begin{equation}\label{est4}
		\begin{array}{rcl}
		\int_B |\nabla u|^2(\ve^2+|\nabla u|^2)^{\frac{q-p-2}{2}} (v-k)^2_{+} \eta^2\,dx &\leq & \int_B |\nabla u|^2(\ve^2+|\nabla u|^2)^{\frac{q-2}{2}} (v-k)_{+} \eta^2\,dx\\
		&= & \int_B \iprod{\nabla u}{\nabla u}(\ve^2+|\nabla u|^2)^{\frac{q-2}{2}} (v-k)_{+} \eta^2\,dx\\
		&\leq & C||u||_{L^\infty(B)}\biggl( \underbrace{\int_B |D^2 u|(\ve^2+|\nabla u|^2)^{\frac{q-2}{2}} (v-k)_{+} \eta^2\,dx}_{I} \\
		&&\qquad +\underbrace{\int_B |\nabla u|^2(\ve^2+|\nabla u|^2)^{\frac{q-4}{2}}|D^2 u| (v-k)_{+} \eta^2\,dx}_{II}\\
		&&\qquad+\underbrace{\int_B |\nabla u|(\ve^2+|\nabla u|^2)^{\frac{q-2}{2}} \nabla((v-k)_{+}) \eta^2\,dx}_{III}\\
		&& \qquad +\underbrace{\int_B |\nabla u|(\ve^2+|\nabla u|^2)^{\frac{q-2}{2}} (v-k)_{+} \eta |\nabla\eta|\,dx}_{IV} \biggr).
		\end{array}
	\end{equation}

The terms in~\cref{est4} containing the Hessian will be estimated using Young's inequality and absorbed to the left hand side. For $\tau>0$ to be chosen, we have
\begin{equation}\label{estI}\begin{array}{rcl}
	I&\leq &\tau(r-\rho)^2\int_B |D^2 u|^2(\ve^2+|\nabla u|^2)^{\frac{p-2}{2}} (v-k)_{+} \eta^4\,dx + \frac{C}{\tau (r-\rho)^2}\int_B (\ve^2+|\nabla u|^2)^{\frac{2q-p-2}{2}} (v-k)_{+} \,dx\\
	&\leq &\tau(r-\rho)^2\int_B |D^2 u|^2(\ve^2+|\nabla u|^2)^{\frac{p-2}{2}} (v-k)_{+} \eta^4\,dx + \frac{C}{\tau (r-\rho)^2}\int_B (\ve^2+|\nabla u|^2)^{\frac{2q-2}{2}}\,\chi_{\{v>k\}}\,dx\\
	&\leq & \tau(r-\rho)^2\int_B |D^2 u|^2(\ve^2+|\nabla u|^2)^{\frac{p-2}{2}} (v-k)_{+} \eta^4\,dx + \frac{C}{\tau (r-\rho)^2}\left(\int_B v^{\al}\,\chi_{\{v>k\}} + 1\,dx\right),
\end{array}\end{equation}
where we use H\"older inequality and \cref{eq5.7}

Since $II\leq I$, it is majorized in the same way as \cref{estI}.

For $III$, we proceed analogously to get
\begin{equation*}
	\begin{array}{rcl}
	III &\leq & \tau(r-\rho)^2\int_B |\nabla (v-k)_{+}|^2 \eta^4\,dx + \frac{C}{\tau (r-\rho)^2}\int_B |\nabla u|^2(\ve^2+|\nabla u|^2)^{q-2} \,\chi_{\{v>k\}}\,dx\\
	&\leq & \tau(r-\rho)^2\int_B |\nabla (v-k)_{+}|^2 \eta^4\,dx + \frac{C}{\tau (r-\rho)^2}\int_B (\ve^2+|\nabla u|^2)^{q-1} \,\chi_{\{v>k\}}\,dx\\
	&\leq& \tau(r-\rho)^2\int_B |\nabla (v-k)_{+}|^2 \eta^4\,dx + \frac{C}{\tau (r-\rho)^2}\left(\int_B v^{\al} \,\chi_{\{v>k\}}+1\,dx\right),\end{array}
\end{equation*}
where we use H\"older inequality noting the restriction \cref{eq5.7}.

For $IV$, we again apply Young's inequality to get
\begin{equation*}
	\begin{array}{rcl}
	IV &\leq&  \tau (r-\rho)^2 \int_B |\nabla u|^2 (\ve^2 + |\nabla u|^2)^{\frac{q-p-2}{2}} (v-k)_+^2 \eta^2 \ dx + \frac{1}{\tau(r-\rho)^2}\int_B (\ve^2+|\nabla u|^2)^{\frac{q-2+p}{2}} \,\chi_{\{v>k\}}\,dx\\
& \leq & \tau (r-\rho)^2 \int_B |\nabla u|^2 (\ve^2 + |\nabla u|^2)^{\frac{q-p-2}{2}} (v-k)_+^2 \eta^2 \ dx + \frac{1}{\tau(r-\rho)^2}\left(\int_B v^\alpha \,\chi_{\{v>k\}}+1\,dx\right),
\end{array}\end{equation*}
where we use Young's inequality along with \cref{eq5.7}.

Now, we choose $\tau$  small enough to absorb terms to the left hand side and combine the estimates to complete the proof of the lemma.
\end{proof}

\section{De Giorgi Iteration}\label{section5}

In this section, we will prove a preliminary Lipschitz bound for the regularized minimizers where the constants do not depend on the regularizing parameter. This is a standard De Giorgi iteration method adapted to the setting of Potential estimates.

\begin{proposition}\label{lipboundone}
	Let $u$ be the solution to~\cref{regminim} and $B_{R_0}(x_0)\Subset B$ be a ball. Let $L$ be such that $||Du||_{L^\infty(B_{R_0}(x_0))}\leq L$ and further suppose that $\al \in \lbr[[] 2, \frac{2N}{N-2}\rbr[)]$. Then the following estimate holds
	\begin{align*}
		||v||_{L^\infty(B_{R_0/2}(x_0))}\leq C\left(\fint_{B_{R_0}(x_0)} v^\alpha \,dx+1\right)^{\frac{1}{\alpha\left(1-\frac N2 +\frac{N}{\alpha}\right)}}+C L\,\left(\mathbf{P}_1^f(x_0,2R_0)\right)^{\frac{1}{\left(1-\frac N2 +\frac{N}{\alpha}\right)}},
			\end{align*} for a constant $C$ depending on $R_0$, $N$, $p$, $||U||_{L^\infty(2B)},||f||_{L^N(B)}$, $M$ and $m$ but independent of $L,\ve$.
\end{proposition}

\begin{proof}
	We start by defining a sequence of decreasing radii $r_j=\dfrac{R_0}{2}+\dfrac{R_0}{2^{j+1}}$, $j=0,1,2,\ldots$ so that $r_0=R_0$ and $r_\infty=\dfrac{R_0}{2}$. Also define $k_j=k-\dfrac{k}{2^{j}}$ so that $k_0=0$ and $k_\infty=k$. We define a sequence of test functions $\eta_j\in C_c^\infty(B_{r_j}(x_0))$ with the property that $0\leq\eta\leq 1$, $\eta=1$ on $B_{r_{j+1}}(x_0)$ and $|\nabla \eta_j|\leq\dfrac{C}{(r_j-r_{j+1})^2}$. With these choices, the improved Caccioppoli inequality from \cref{cacc3} becomes 
	\begin{align*}
		\int_{B_{r_{j}}(x_0)}|\nabla( (v-k_{j+1})_{+}\eta_j^2)|^2\,dx\leq \lbr \frac{C}{(r_j-r_{j+1})^4} + \frac{C}{(r_j-r_{j+1})^2}\rbr&\left(\int_{B_{r_j}(x_0)} v^{\al} \chi_{\{v>k_{j+1}\}} +1\,dx\right)\\
		&+C L^2\int_{B_{r_j}(x_0)} f^2\,dx,
	\end{align*}
where we have taken the test function inside the gradient on the left hand side of \cref{cacc3} and made use of $\al \geq 2$.

	We transform the above equation by an application of Sobolev's inequality for $p=2$ on the LHS as follows: Let $\theta=\frac{N}{N-2}$, then we have
	\begin{equation}\label{cacc6}\begin{array}{rcl}
		\left(\fint_{B_{r_{j+1}}(x_0)} (v-k_{j+1})_{+}^{2\theta}\,dx\right)^{1/\theta} &\leq &\lbr  \frac{C\,r_j^2}{(r_j-r_{j+1})^4} + \frac{C\,r_j^2}{(r_j-r_{j+1})^2}\rbr \left(\fint_{B_{r_j}(x_0)} v^{\al} \chi_{\{v>k_{j+1}\}}\,dx+1\right)\\
		&& +C \,r_j^2 L^2\fint_{B_{r_j}(x_0)} f^2\,dx.
	\end{array}\end{equation}
	Now, notice that 
	\begin{align*}
		\frac{r_j^2}{(r_j-r_{j+1})^4}+ \frac{r_j^2}{(r_j-r_{j+1})^2}\leq C\,b^j
	\end{align*} for some $b>1$ and $C$ depending on $R_0$.
	Similarly
	\begin{align*}
		{r_j^2}= \frac{R_0^2}{4}\left(1+\frac{1}{2^{j}}\right)^2=\left(\frac{R_0}{2^{j+1}}\right)^2(2^j+1)^2\leq C\,b^j\,(r_{j-1}-r_j)^2
	\end{align*} for some $b>1$ and $C$ depending on $R_0$, since $(r_{j-1}-r_j)=\frac{R_0}{2^{j+1}}$.
	Substituting these estimations in~\cref{cacc6}, we get
	\begin{equation*}
		\left(\fint_{B_{r_{j+1}}(x_0)} (v-k_{j+1})_{+}^{2\theta}\,dx\right)^{1/\theta}\leq C\,b^j\left\{\left(\fint_{B_{r_j}(x_0)} v^{\al} \chi_{\{v>k_{j+1}\}}\,dx+1\right)+ \,(r_{j-1}-r_j)^2 L^2\fint_{B_{r_j}(x_0)} f^2\,dx\right\}.
	\end{equation*} 
	We also have the following estimate:
	\begin{equation}
		\begin{array}{rcl}
		\int_{B_{r_j}}(v-k_j)_+^{\al}\,dx&\geq &\int_{B_{r_j}}(v-k_j)_+^{\al}\chi_{\{v>k_{j+1}\}}\,dx
		\geq \int_{B_{r_j}}v^{\al}\left(1-\frac{2^{j+1}-2}{2^{j+1}-1}\right)^{\al}\chi_{\{v>k_{j+1}\}}\,dx\\
		&\geq & \frac{C}{2^{{\al}(j+1)}}\int_{B_{r_j}}v^{\al}\chi_{\{v>k_{j+1}\}}\,dx, 
\end{array}	\end{equation}
	so that
	\begin{align}\label{cacc8}
		\left(\fint_{B_{r_{j+1}}(x_0)} (v-k_{j+1})_{+}^{2\theta}\,dx\right)^{1/\theta}\leq C\,b^j\left\{\left(\fint_{B_{r_j}(x_0)} (v-k_{j})_+^{\al}\,dx+1\right)+ \,(r_{j-1}-r_j)^2 L^2\fint_{B_{r_j}(x_0)} f^2\,dx\right\},
	\end{align}

	Now, define $V_j:=\left(\fint_{B_{r_j}(x_0)}(v-k_{j})^{\al}_+\,dx\right)^{1/\al}$ and $W_j:=(r_{j-1}-r_j)\left(\fint_{B_{r_j}(x_0)}f^2\,dx\right)^{1/2}$.

	We are now ready to write an iterative estimate for $V_j$. Observe that from the restriction $\al < \frac{2N}{N-2}$, we have
	\begin{align*}
		V_{j+1}=\left(\fint_{B_{r_{j+1}}(x_0)}(v-k_{j+1})^{\al}\right)^{1/\al}\leq C\left(\fint_{B_{r_{j+1}}(x_0)}(v-k_{j+1})^{2\theta}\right)^{1/2\theta}\left(\frac{|\{v>k_{j+1}\}|}{|B_{r_j}|}\right)^{(2\theta-\al)/2\tht\al}.
	\end{align*}
	In order to estimate the last term appearing on the right hand side, we have 
	\begin{align*}
		\int_{B_{r_j}}(v-k_j)_+^{\al}\chi_{\{v>j_{j+1}\}}\,dx\geq (k_{j+1}-k_j)^{\al}|\{v>k_{j+1}\}|=\frac{k^{\al}}{2^{{\al}(j+1)}}|\{v>k_{j+1}\}|
	\end{align*} so that
	\begin{align}\label{estjplus11}
		V_{j+1}=\left(\fint_{B_{r_{j+1}}(x_0)}(v-k_{j+1})^{\al}\right)^{1/{\al}}\leq C\left(\fint_{B_{r_{j+1}}(x_0)}(v-k_{j+1})^{2\theta}\right)^{1/2\theta}\left(\frac{{b}^{{\al}(j+1)}}{k^{\al}}V_j^{\al}\right)^{(2\theta-{\al})/2\tht{\al}}.
	\end{align}
	On substituting \cref{cacc8} in \cref{estjplus11}, we get 
	\begin{equation*}
		\begin{array}{rcl}
		V_{j+1}&\leq & C\frac{b^j}{k^{\frac{2\tht-\alpha}{2\tht}}}\left(V_j^{\frac{\alpha}{2}}+1+L\,W_j\right)\left(V_j\right)^{\frac{2\tht-\alpha}{2\tht}}\\
		&\leq & C\frac{b^j}{k^{\frac{2\tht-\alpha}{2\tht}}}\left(V_j+1+L\,W_j\right)^{\frac{\alpha}{2}+\frac{2\tht-\alpha}{2\tht}}
		= C\frac{b^j}{k^{\frac{2\tht-\alpha}{2\tht}}}\left(V_j+1+L\,W_j\right)^{1+\frac{\alpha}{N}}\\
		&\leq & C\frac{b^j}{k^{\frac{2\tht-\alpha}{2\tht}}}\left(C\frac{b^{j-1}}{k^{\frac{2\tht-\alpha}{2\tht}}}\left(V_{j-1}+1+L\,W_{j-1}\right)^{1+\alpha/N}+1+L\,W_j\right)^{1+\alpha/N}\\
		&\leq & \left(\frac{C}{k^{\frac{2\tht-\alpha}{2\tht}}}\right)^{1+(1+\alpha/N)}b^{j+(j-1)(1+\alpha/n)}\left(V_{j-1}+1+L(W_{j-1}+W_j)\right)^{(1+\alpha/N)^2}\\
		&\vdots &\\
		&\leq & \left(\frac{C}{k^{\frac{2\tht-\alpha}{2\tht}}}\right)^{\frac{(1+\alpha/N)^j-1}{\alpha/N}}b^{\frac{(1+\alpha/N)^{j+1}+1}{(\alpha/N)^2}}\left(V_0+1+L\sum_{k=0}^{j}W_{j-k}\right)^{(1+\alpha/N)^j}.
	\end{array}\end{equation*}
	A standard iteration lemma implies $V_\infty=0$ provided $V_0\leq \left(\frac{C}{k^{\frac{2\tht-\alpha}{2\tht}}}\right)^{-N/\alpha}\,b^{-N^2/\alpha^2}$ is satisfied. This requires
	\begin{align*}
		k^{1-\frac{N}{2}+\frac{N}{\alpha}}=C^{N/\alpha}b^{N^2/\alpha^2}\left(V_0+1+L\sum_{j=1}^\infty W_j\right), 
	\end{align*} assuming that $\sum_{j=1}^\infty W_j<\infty$.
	Notice that,
	\begin{equation*}\begin{array}{rcl}
		\sum_{j=1}^\infty W_j&=&\sum_{j=1}^\infty (r_{j-1}-r_j)\left(\fint_{B_{r_j}(x_0)}f^2\,dx\right)^{1/2}\\
		&=&\sum_{j=1}^\infty \int_{r_j}^{r_{j-1}}\rho\,d\rho\left(\fint_{B_{r_j}(x_0)}f^2\,dx\right)^{1/2}\\
		&\leq& 2^{N/2}\sum_{j=1}^\infty \int_{r_j}^{r_{j-1}}\left(\rho^2\fint_{B_{\rho}(x_0)}f^2\,dx\right)^{1/2}\frac{d\,\rho}{\rho}\\
		&\leq& 2^{N/2} \int_{R_0/2}^{3R_0/2}\left(\rho^2\fint_{B_{\rho}(x_0)}f^2\,dx\right)^{1/2}\frac{d\,\rho}{\rho}\\
		&\leq& 2^{N/2} \int_{0}^{2R_0}\left(\rho^2\fint_{B_{\rho}(x_0)}f^2\,dx\right)^{1/2}\frac{d\,\rho}{\rho}=2^{N/2}\mathbf{P}_1^f(x_0,2R_0).
	\end{array}\end{equation*}
	Since $v\leq k$ in $B_{R_0/2}(x_0)$, we get the result.
\end{proof}

\section{Interpolation estimates}\label{section6}

In this section, we will prove a Lipschitz bound for the regularized minimizers where the right hand side is in terms of $L^p$ norm of $\nabla u$. This is achieved by a standard interpolation.

\begin{proposition}\label{lipboundtwo}
	Let $u_n$ be the solution to~\cref{regminim}. Let $B_R\Subset B$ be a ball and denote $f_{B_R}=f\mathbbm{1}_{B_R}$ so that $f_{B_R}\in L^2(\mathbb{R}^N)$. Further suppose that 
	\begin{align}\label{restrict3}
		\frac{2(q-1)}{p}\leq \alpha <2+\min\left\{\frac{2}{N},\frac{4(p-1)}{Np+2-2p}\right\}.
	\end{align}
	Let $F_n$ be a sequence of regularized integrands as in~\cref{mainprob2}. Then it holds that
	\begin{align}\label{cacc9}
		||\nabla u_n||^p_{L^\infty(B_{R/2})}\leq C\left(\fint_{B_{R}} F_n(\nabla u_n) \,dx\right)^{\frac{2}{2\alpha-\alpha N+2N}}+C+C\left( ||\mathbf{P}_1^{f_{B_R}}(\cdot,2R_0)||_{L^\infty(B_R)}\right)^{\frac{2p\alpha}{p(2\alpha-\alpha N+2N)-2\alpha}},
			\end{align} for a constant $C$ depending on $R_0$, $N$, $p$, $||U||_{L^\infty(2B)},||f||_{L^N(B)}$, $M$ and $m$ but independent of $\ve$.
\end{proposition}
\begin{proof}
	Consider concentric balls $B_{R/2}\Subset B_s\Subset B_t\Subset B_R$, a point $x_0\in B_s$ and $R_0=t-s$ such that $B_{R_0}(x_0)\subset B_t$, then we have $||\nabla u||_{L^\infty(B_{R_0}(x_0))}\leq ||\nabla u||_{L^\infty(B_t)}=L$, so that
	\begin{align*}
		||v||_{L^\infty(B_s)}&\leq C\left(\fint_{B_{R_0}(x_0)} v^\alpha \,dx\right)^{\frac{1}{\alpha\left(1-\frac N2 +\frac{N}{\alpha}\right)}}+C ||\nabla u||_{L^\infty(B_t)}^{\frac{1}{\left(1-\frac N2 +\frac{N}{\alpha}\right)}}\,\mathbf{P}_1^{f_{B_R}}(x_0,2R_0)^{\frac{1}{\left(1-\frac N2 +\frac{N}{\alpha}\right)}}+C\\
		&\leq C\left(\fint_{B_{R_0}(x_0)} v^\alpha \,dx\right)^{\frac{1}{\alpha\left(1-\frac N2 +\frac{N}{\alpha}\right)}}+C ||\nabla u||_{L^\infty(B_t)}^{\frac{1}{\left(1-\frac N2 +\frac{N}{\alpha}\right)}}\,||\mathbf{P}_1^{f_{B_R}}(\cdot,2(t-s))||_{L^\infty(B_t)}^{\frac{1}{\left(1-\frac N2 +\frac{N}{\alpha}\right)}}+C\\
		&\leq C\left(\fint_{B_t} v^\alpha \,dx\right)^{\frac{1}{\alpha\left(1-\frac N2 +\frac{N}{\alpha}\right)}}+C ||(\ve_n+|\nabla u|^2)^{1/2}||_{L^\infty(B_t)}^{\frac{1}{\left(1-\frac N2 +\frac{N}{\alpha}\right)}}\,||\mathbf{P}_1^{f_{B_R}}(\cdot,2(t-s))||_{L^\infty(B_t)}^{\frac{1}{\left(1-\frac N2 +\frac{N}{\alpha}\right)}}+C\\
		&\leq C||v||^{\frac{\alpha-1}{\alpha\left(1-\frac N2 +\frac{N}{\alpha}\right)}}_{L^\infty(B_t)}\left(\fint_{B_t} v \,dx\right)^{\frac{1}{\alpha\left(1-\frac N2 +\frac{N}{\alpha}\right)}}+C ||v||_{L^\infty(B_t)}^{\frac{1}{p\left(1-\frac N2 +\frac{N}{\alpha}\right)}}\,||\mathbf{P}_1^{f_{B_R}}(\cdot,2(t-s))||_{L^\infty(B_t)}^{\frac{1}{\left(1-\frac N2 +\frac{N}{\alpha}\right)}}+C\\
		&\leq \frac{1}{2}||v||_{L^\infty(B_t)}+\frac{C}{(t-s)^{\frac{2N}{2+2N-\alpha N}}}\left(\int_{B_t} v \,dx\right)^{\frac{2}{2+2N-\alpha N}}+C ||\mathbf{P}_1^{f_{B_R}}(\cdot,R)||_{L^\infty(B_R)}^{\frac{2p\alpha}{p(2\alpha-\alpha N+2N)-2\alpha}}+C,
	\end{align*} where to obtain the last inequality, we have used Young's inequality along with the restrictions in \cref{restrict3}. Now, the result follows from an application of \cref{iterlemma}.
\end{proof}

\section{Passage to limit in the regularization process}\label{section7}

In this section, we will pass to the limit in \cref{cacc9} as $n\to\infty$.  
\begin{proof}[Proof of \cref{maintheorem}]
	We begin with the following estimates where $\gamma*=\dfrac{Np}{N-p}$ for $p<N$ and $\gamma*=\dfrac{2N}{N-1}$ otherwise. 
\begin{align*}
	\fint_B|f_n(u_n-U_n)|\,dx&\leq \left(\fint |f_n|^N\,dx\right)^{1/N}\left(\fint|u_n-U_n|^{\frac{N}{N-1}}\right)^{\frac{N-1}{N}}\nonumber\\
	&\leq \left(\fint |f_n|^N\,dx\right)^{1/N}\left(\fint|u_n-U_n|^{\gamma*}\right)^{1/\gamma*}\nonumber\\
	&\leq C\left(\fint |f_n|^N\,dx\right)^{1/N}\left(\fint|\nabla u_n-\nabla U_n|^{p}\right)^{1/p}\nonumber\\
	&\leq C\left(\fint |f_n|^N\,dx\right)^{1/N}\left(\fint_B F_n(\nabla u_n)\,dx+\fint_B F_n(\nabla U_n)\,dx\right)^{1/p}\nonumber\\
	&\leq C||f_n||^{\frac{p}{p-1}}_{L^N(B)}+\left(\fint_B F_n(\nabla u_n)\,dx+\frac{|B+\ve_n B_1|}{|B|}\fint_{B+\ve_n B_1} F_n(\nabla U)\,dx\right)^{1/p}\nonumber\\
	&\leq \frac{1}{2}\fint_B F_n(\nabla u_n)\,dx+\frac{|B+\ve_n B_1|}{2|B|}\fint_{B+\ve_n B_1} F_n(\nabla U)\,dx+C||f_n||^{\frac{p}{p-1}}_{L^N(B)},
\end{align*}
where the second inequality is due to H\"older's inequality, the third inequality is by Sobolev's inequality, the fifth inequality is by Jensen's inequality and the sixth inequality is by Young's inequality.

On the other hand, by minimality of $u_n$, we have
\begin{align*}
	\fint_B F_n(\nabla u_n)\,dx&\leq \fint_B F_n(\nabla U_n)+\fint_B f_n(u_n-U_N)\,dx\nonumber\\
	&\leq \frac{|B+\ve_n B_1|}{|B|}\fint_{B+\ve_n B_1} F_n(\nabla U)\,dx+\fint_B f_n(u_n-U_N)\,dx\nonumber\\
	&\leq \frac{1}{2}\fint_B F_n(\nabla u_n)\,dx+C\fint_{B+\ve_n B_1} F_n(\nabla U)\,dx+C||f_n||^{\frac{p}{p-1}}_{L^N(B)}.
\end{align*}
After absorbing the first term to the left, we get,
\begin{align}\label{est11}
	\fint_B F_n(\nabla u_n)\,dx\leq C\fint_{B+\ve_n B_1} F_n(\nabla U)\,dx+C||f_n||^{\frac{p}{p-1}}_{L^N(B)}.
\end{align}
Combining \cref{est11} and \cref{cacc9}, we get
\begin{align*}
	||\nabla u||^p_{L^\infty(B_{R/2})}\leq & C\left\{\frac{|B|}{|B_R|}\left(\fint_{B+\ve_n B_1} F_n(\nabla U) \,dx+||f_n||^{\frac{p}{p-1}}_{L^N(B)}\right)\right\}^{\frac{2}{2+2N-\alpha N}}\\
	&\qquad+C+ C||\mathbf{P}_1^{f_{B_R}}(\cdot,2R_0)||_{L^\infty(B_R)}^{\frac{2p\alpha}{p(2\alpha-\alpha N+2N)-2\alpha}},
\end{align*} where the constant only depends on $R_0$, $N$, $p$, $||U||_{L^\infty(2B)},||f||_{L^N(B)}$, $M$ and $m$.
However, by definition of $F_n$, we have 
\begin{align}\label{est12}
	F_n(z)\lesssim F(z)+\ve_n^p,
\end{align}
where the constant does not depend on $n$. Therefore, we can write
\begin{align*}
||\nabla u||^p_{L^\infty(B_{R/2})}\leq & C\left\{\frac{|B|}{|B_R|}\left(\fint_{B+\ve_n B_1} F(\nabla U) \,dx+1+||f||^{\frac{p}{p-1}}_{L^N(B)}\right)\right\}^{\frac{2}{2+2N-\alpha N}}\\
&\qquad+C+ C||\mathbf{P}_1^{f_{B_R}}(\cdot,2R_0)||_{L^\infty(B_R)}^{\frac{2p\alpha}{p(2\alpha-\alpha N+2N)-2\alpha}},
\end{align*}
Also,
\begin{align*}
	||\mathbf{P}_1^{f_{B_R}}(\cdot,R)||_{L^\infty(B_R)}\leq C||f_{B_R}||_{L(N,1)(B_{2R})}=C||f||_{L(N,1)(B_{R})},
\end{align*}
therefore,
\begin{align}
	\label{cacc12}
	||\nabla u||^p_{L^\infty(B_{R/2})}\leq & C\left\{\frac{|B|}{|B_R|}\left(\fint_{B+\ve_n B_1} F(\nabla U) \,dx+||f||^{\frac{p}{p-1}}_{L^N(B)}\right)\right\}^{\frac{2}{2+2N-\alpha N}}\nonumber\\
	&\qquad+C+ C||f||_{L(N,1)(B_R)}^{\frac{2p\alpha}{p(2\alpha-\alpha N+2N)-2\alpha}},
\end{align}
Since $||F_n(\nabla U)||_{L^1(B+\ve_n B_1)}$ is bounded by \cref{est12}, therefore so is $||F_n(\nabla u_n)||_{L^1(B)}$ by \cref{est11}. Therefore, $u_n$ is bounded in $W^{1,p}(B)$. From \cref{cacc12}, $u_n$ is bounded in $W^{1,\infty}(B_{R/2})$. As a result, there exists $u_0\in U+W_0^{1,p}(B)$ such that
\begin{align*}
	u_n\rightharpoonup u_0 \mbox{ in }W^{1,p}(B)\mbox{-weak}\\
	u_n\to u_0 \mbox{ in }L^{\frac{N}{N-1}}(B)\\
	u_n\rightharpoonup u_0 \mbox{ in }W^{1,p}(B_{R/2})-\mbox{weak-}*.
\end{align*}
By lower-semicontinuity of norm, we have 
\begin{align*}
	||\nabla u_0||^p_{L^\infty(B_{R/2})}\leq\liminf_{n\to\infty}||\nabla u_n||^p_{L^\infty(B_{R/2})}\leq C&\left\{\frac{|B|}{|B_R|}\left(\fint_{B} F(\nabla U) \,dx+||f||^{\frac{p}{p-1}}_{L^N(B)}\right)\right\}^{\frac{2}{2+2N-\alpha N}}\\
	&\qquad+C+ C||f||_{L(N,1)(B_R)}^{\frac{2p\alpha}{p(2\alpha-\alpha N+2N)-2\alpha}},
\end{align*}
It remains to show that $u_0=U$. We will show that $u_0$ is also a solution to \cref{mainprob}.
Observe that, by lower semicontinuity of norms,
\begin{align*}
	\int_{B} F(\nabla u_0)\,dx\leq \liminf_{n\to\infty}\int_B F(\nabla u_n)\,dx=\liminf_{n\to\infty}\int_B F_n(\nabla u_n)\,dx,
\end{align*}
since 
\[\lim_{n\to\infty}\int_B F(\nabla u_n)\,dx=\lim_{n\to\infty}\int_B F_n(\nabla u_n)\,dx\] due to the fact that $F_n$ converges to $F$ on compact sets and $\nabla u_n$ is bounded independent of $n$ by \cref{cacc9}.
 
Hence, by minimality of $u_n$ for $\mathfrak{F}_n$, the strong convergence of $U_n\to U$, and that of $u_n\to u_0$ in $L^{\frac{N}{N-1}}$, and the strong convergence of $f_n\to f$ in $L^N$, we have
\begin{align*}
	\mathfrak{F}(u_0,B)\leq \liminf_{n\to\infty}\mathfrak{F}_n(u_n,B)\leq \liminf_{n\to\infty}\mathfrak{F}_n(U_n,B)\leq \lim_{n\to\infty}\int_{B+\ve_n B_1}F_n(\nabla U)-\lim_{n\to\infty}f_nU_n\,dx=\mathfrak{F}(\nabla U,B).
\end{align*}
By strong convexity of $\mathfrak{F}$, we have $u_0=U$.
\end{proof}


\end{document}